\documentclass[12pt, twoside]{article}
\usepackage{amsmath,amsthm,amssymb}
\usepackage{times}
\usepackage{enumerate}
\usepackage{bm}
\usepackage[all]{xy}
\usepackage{mathrsfs}
\usepackage{amscd}
\usepackage{color}
\usepackage{hyperref}
\hypersetup{hypertex=true,
	colorlinks=true,
	linkcolor=red,
	anchorcolor=red,
	citecolor=red}
\def\titlerunning#1{\gdef\titrun{#1}}
\makeatletter
\def\author#1{\gdef\autrun{\def\and{\unskip, }#1}\gdef\@author{#1}}
\def\address#1{{\def\and{\\\hspace*{18pt}}\renewcommand{\thefootnote}{}%
		\footnote {#1}}%
	\markboth{\autrun}{\titrun}}
\makeatother
\def\email#1{e-mail: #1}

\def\keywords#1{\par\medskip
	\noindent\textbf{Keywords.} #1}

\newtheorem{theorem}{Theorem}[section]
\newtheorem{corollary}[theorem]{Corollary}
\newtheorem{lemma}[theorem]{Lemma}
\newtheorem{proposition}[theorem]{Proposition}
\theoremstyle{definition}
\newtheorem{definition}[theorem]{Definition}
\newtheorem{remark}[theorem]{Remark}

\numberwithin{equation}{section}

\xyoption{all}

\def \C {\mathbb{C}}

\def \a {\alpha }
\def \b {\beta}

\def \de {\delta}
\def \De {\Delta}
\def \la {\lambda}
\def \La {\Lambda}
\def\w {\omega}
\def\Om{\Omega}

\def\pa{\partial}
\def\na {\nabla}

\begin{document}
	
\baselineskip=17pt

\titlerunning{On harmonic symmetries for locally conformally K\"{a}hler manifolds}
\title{On harmonic symmetries for locally conformally K\"{a}hler manifolds}

\author{Teng Huang}

\date{}

\maketitle

\address{Teng Huang: School of Mathematical Sciences, University of Science and Technology of China; CAS Key Laboratory of Wu Wen-Tsun Mathematics,  University of Science and Technology of China, Hefei, Anhui, 230026, People’s Republic of China; \email{htmath@ustc.edu.cn;htustc@gmail.com}}
\begin{abstract}
In this article, we study harmonic symmetries on the compact locally conformally K\"{a}hler manifold $M$ of $dim_{\C}=n$. The space of harmonic symmetries is a subspace of harmonic differential forms which defined by the kernel of a certain Laplacian-type operator $\square$. We observe that the spaces $\ker(\square)\cap\Om^{l}=\{0\}$ for any $|l-n|\geq2$ and $\ker\De_{\bar{\pa}}\cap P^{k,n-1-k}\cap\ker(i_{\theta^{\sharp}})\cong\ker(\square^{k,n-1-k})$, $\ker\De_{\bar{\pa}}\cap P^{k,n-k}\cong\ker(\square^{k,n-k})$. Furthermore, suppose that $M$ is a Vaisman manifold, we prove that
(i) $\a$ is $(n-1)$-form in $\ker(\square)$ if only if  $\a$ is a transversally harmonic and transversally effective $\mathcal{V}$-foliate form; (ii) $\a$ is a $(p,n-p)$-form in $\ker(\square^{p,n-p})$ if only if
there are two forms $\b_{1}\in\mathcal{S}^{p-1,n-p}$ and $\b_{2}\in\mathcal{S}^{p,n-p-1}$ such that $\a=\theta^{1,0}\wedge\b_{1}+\theta^{0,1}\wedge\b_{2}$.
\end{abstract}
\section{Introduction}
On a compact complex manifold, one can consider two different kinds of invariants: the topological ones of the underlying (real) manifold and the complex ones. Among the first ones, a fundamental role is played by de Rham cohomology, and among the second ones, we recall the Dolbeault, Bott–Chern and Aeppli cohomologies. The dimensions of Dolbeault, Bott-Chern and Aeppli cohomologies which are bounded from below by topological quantities, such as Betti numbers (see \cite{AT13,AT15}). But there are comparatively few reverse inequalities, and these are desirable for showing that a complex manifold has non-trivial topology, or conversely that a smooth manifold does not have a complex structure.

We denote by $M$ a compact Hermitian manifold $M$, that is, $M$ is a compact complex manifold with compatible metric. In \cite{Wil20,Wil}, Wilson described some topological and geometric inequalities for $M$. They are expressed in terms of the kernel of a certain Laplacian-type operator $$\square=\De_{\pa}+\De_{\bar{\pa}}+\De_{\tau}+\De_{\bar{\tau}}+\De_{\la}+\De_{\bar{\la}}.$$ 
This is a real operator, and the last four summands are all order zero. The kernel of this second-order self-adjoint elliptic operator $\square$ determines a subspace of $\De_{d}$-harmonic forms that satisfies the Serre, Hodge, and conjugation dualities, generalizing the K\"{a}hler case. Moreover,  there is an induced representation of $\mathfrak{sl}(2,\mathbb{C})$ on the harmonic forms on $\ker(\square)$, yielding a generalization of hard Lefschetz duality (see \cite[Theorem 3.1]{Wil20}). The result relies on a generalization of the K\"{a}hler identities to the Hermitian setting \cite{Dem86,Gri66}.

Define the space of $\square$-harmonic forms in degree $k$ by letting
$$\mathcal{H}_{\square}^{k}=\ker(\square)\cap\Om^{k}$$
where $\Om^{k}$ denotes the space of $k$-forms.  We denote by $h^{k}_{\square}$ the dimension of $\mathcal{H}_{\square}^{k}$. It remains to further study what the dimensions of those subspaces tell us about a given Hermitian structure, and conversely, to determine what are the permissible numbers for a given complex structure.

In this article, we are interested in topological and complex analytic properties of compact LCK manifolds. A locally conformally K\"{a}hler (LCK) manifold is a Hermitian manifold whose metric is conformal to a K\"{a}hler metric in some neighbourhood of every point. This definition is equivalent to the existence of a global closed one-form $\theta$ (called the $\textbf{Lee}$ form) such that the fundamental two-form $\w$ satisfies  $d\w=\theta\wedge\w$ (see \cite{DO98,Tsu94,Vai82}). These manifolds appear naturally in complex geometry. Most examples of compact non-K\"{a}hler manifolds studied in complex geometry admit an LCK structure. An interesting example of such manifolds is offered by the Hopf manifolds \cite{Vai76}, which are compact and have no K\"{a}hler metric at all.  In many situations, the LCK structure becomes useful for the study of topology and complex geometry of an LCK manifold. Throughout our article, we always say that an LCK manifold $M$ cannot  admit any Kahler metric, that is, $M$ is a non-K\"{a}herian manifold.
\begin{theorem}[=Theorem \ref{T4} and \ref{P4}]
Let $(M,J,\theta)$ be a compact locally conformally K\"{a}hler manifold of $dim_{\C}=n$. Then for any $|l-n|\geq 2$, we have
$$\mathcal{H}_{\square}^{l}=\{0\}.$$
Furthermore,
\begin{equation*}
\ker\De_{\bar{\pa}}\cap P^{k,n-1-k}\cap\ker(i_{\theta^{\sharp}})\cong\ker(\square^{k,n-1-k}),
\end{equation*}
\begin{equation*}
\ker\De_{\bar{\pa}}\cap P^{k,n-k}\cong \ker(\square^{k,n-k}).
\end{equation*}
\end{theorem}
If $(M,J)$ is a compact complex manifold, then $\dim\ker(\square^{0,0})=1$ if and only if $M$ is K\"{a}hlerian (resp. $\dim\ker(\square^{0,0})=0$ if and only if $M$ is non-K\"{a}hlerian) (see \cite{Wil20}). In \cite{Vai79}, the author have given some sufficient conditions for a compact LCK manifold to be K\"{a}hler. Following Theorem \ref{T4}, we can give a sufficient condition to such that a compact complex non-K\"{a}hlerian manifold admits a Hermitian metric $g$ which is not an LCK metric.
\begin{corollary}
Let $(M,J,g)$ be a compact, complex, non-K\"{a}hlerian manifold of $dim_{\C}=n$. If there is a constant $k$, $|k-n|\geq 2$, such that  $h^{k}_{\square}\geq 1$, then $g$ is not an locally conformally K\"{a}hler metric. 
\end{corollary}
Among the LCK manifolds, a distinguished class is the Vaisman manifold. The topology of compact Vaisman manifolds is very different from that of K\"{a}hler manifolds. In \cite{OV19}, the authors studied harmonic forms and Hodge decomposition on Vaisman and Sasakian manifolds. In \cite{Vai82}, Vaisman established a Hodge decomposition theorem of $\De_{d}$-harmonic $k$-form for any $k\leq n-1$. In \cite{Tsu94}, Tsukada obtained a complex version Hodge decomposition theorem on Vaisman manifold. But none of them give a Hodge decomposition of $\De_{d}$-harmonic $n$-form. Here, we study the Hodge decomposition of the forms in $\ker\square\cap\Om^{k}$, $k=n-1,n$, on Vaisman manifold. We will develop a version of Hodge theory on the Vaisman manifold, hence we can give a Hodge decomposition of the $\De_{d}$-harmonic $n$-forms.
\begin{theorem}[=Theorem \ref{T8} and \ref{T9}]\label{T1}
Let $(M,J,\theta)$ be a compact Vaisman manifold of $dim_{\C}=n$. Then 
$$\ker(\square^{p,n-1-p})\cong\mathcal{S}^{p,n-1-p}(\mathcal{V}),$$
$$
\mathcal{H}_{d}^{n}\cong\mathcal{H}^{n}_{\square}\cong\bigoplus_{k=0}^{n}\mathcal{H}^{k,n-k}_{\bar{\pa}},
$$
where $\mathcal{S}^{p,q}$ denote the vector space of transversally harmonic and transversally effective $\mathcal{V}$-foliate $(p,q)$-forms. 
Furthermore, the following four conditions are equivalent:\\ 
(i) $\a\in\ker(\square^{p,n-p})$, \\
(ii) $\a\in\ker\De_{\bar{\pa}}\cap\Om^{p,n-p}$,\\
(iii) $\a\in\ker\De_{d}\cap\Om^{p,n-p}$,\\
(iv) there are $\b_{1}\in\mathcal{S}^{p-1,n-p}$ and $\b_{2}\in\mathcal{S}^{p,n-p-1}$ such that
$$\a=\theta^{1,0}\wedge\b_{1}+\theta^{0,1}\wedge\b_{2}.$$
\end{theorem} 

\section{Preliminaries}
\subsection{Harmonic symmetries}
Throughout this section $(M,J,g)$ will denote a compact Hermitian manifold of complex dimension $n\geq2$, with fundamental two-form defined by $\w(X,Y)=g(X,JY)$. The space of all smooth $(p,q)$-forms (resp. $k$-forms) on $M$ denoted $\Om^{p,q}(M)$ (resp. $\Om^{k}(M)$. Let $\langle,\rangle$ denote the pointwise inner product. The global inner product is defined $$(\a,\b)=\int_{M}\langle\a,\b\rangle dV,$$
where $dV=\frac{\w^{n}}{n!}$. Another important operator is the operator $L$ of type $(1,1)$ defined by
$$L(\cdot)=\w\wedge\cdot,$$ 
and its adjoint $\La=\ast^{-1}L\ast$:
$$\langle\a,\La\b\rangle=\langle L\a,\b\rangle.$$
\begin{definition}
A differential $k$-form $\a_{k}$ with $k\leq n$ is called primitive, i.e., $\a_{k}\in P^{k}(M)$, if it satisfies the two equivalent conditions: (i) $\La\a_{k}=0$; (ii) $L^{n-k+1}\a_{k}=0$.	
\end{definition} 
One can also prove the following very important result: every $k$-form $\a$ on $M$ has a unique decomposition of the form 
$$\a=\sum_{r\geq\max{(0,k-n)}}L^{r}\b_{k-2r}$$
where all of the $\b_{j}$ are primitive forms of a corresponding degree and we denote $\b_{j}=0$ for $j\notin[0,n]$.

Let $\la=[\pa,L]=(\pa\w\wedge\cdot)$, so that $\bar{\la}=[\bar{\pa},L]=(\bar{\pa}\w\wedge\cdot)$. The operator $\bar{\la}$ has bidegree $(1,2)$ and governs the symplectic condition: a Hermitian manifold is K\"{a}hler if and only if $\bar{\la}=0$. In \cite{Dem86}, Demailly derives a set of Hermitian identities  which generalize the K\"{a}hler identities. Consider the zero-order torsion operator $\tau:=[\La,\la]$ of bidegree $(1,0)$. Demailly shown
\begin{equation}\label{E1}
\begin{split}
&[\La,\bar{\pa}]=-i(\pa^{\ast}+\tau^{\ast}),\\ &[\La,\pa]=i(\bar{\pa}^{\ast}+\bar{\tau}^{\ast}),\\
& [L,\bar{\pa}^{\ast}]=-i(\pa+\tau),\\
& [L,\pa^{\ast}]=i(\bar{\pa}+\bar{\tau})\\
\end{split}
\end{equation}
with K\"{a}hler identities recovered in the case $\tau=0$ and 
\begin{equation}\label{E2}
\begin{split}
&[\La,\tau]=-2i\bar{\tau}^{\ast},\ [L,\bar{\tau}]=3\bar{\la},\\
& [\La,\bar{\tau}]=2i\tau^{\ast},\ [L,\tau]=3\la,\\ &[L,\tau^{\ast}]=-2i\bar{\tau},\ [\La,\bar{\tau}^{\ast}]=-3\bar{\la}^{\ast},\\
&[L,\bar{\tau}^{\ast}]=2i\tau,\ [\La,\tau^{\ast}]=-3\la^{\ast}.\\
\end{split}
\end{equation}
We also have (see \cite{Wil20}) 
\begin{equation}\label{E3}
\begin{split}
&[\La,\la]=\tau,\ [L,\la]=0,\\
&[\La,\bar{\la}]=\bar{\tau},\ [L,\bar{\la}]=0,\\
&[L,\la^{\ast}]=-\tau^{\ast},\ [\La,\la^{\ast}]=0,\\
&[L,\bar{\la}^{\ast}]=-\bar{\tau}^{\ast}, [\La,\bar{\la}^{\ast}]=0.\\
\end{split}
\end{equation}
For any operator $\de$, let $\De_{\de}=[\de,\de^{\ast}]$, and let $\De_{\de}^{p,q}$ denote the restriction to $\Om^{p,q}$. 
\begin{proposition}(\cite[Corollary 2.1]{Wil20})
For any Hermitian manifold there is an induced representation of $\mathfrak{sl}(2,\C)$ on the space $\ker(\De_{\tau}+\De_{\bar{\tau}}+\De_{\la}+\De_{\bar{\la}})$.
\end{proposition}
The operator $\De_{\tau}+\De_{\bar{\tau}}+3\De_{\la}+3\De_{\bar{\la}}$ is the one that commutes directly with $L$ and $\La$ on all forms (see \cite[Equation (4)]{Wil}).Noting that $\ker(\De_{\tau}+\De_{\bar{\tau}}+\De_{\la}+\De_{\bar{\la}})=\ker(\De_{\tau}+\De_{\bar{\tau}}+3\De_{\la}+3\De_{\bar{\la}})$. Hence, we get
\begin{equation*}
	\begin{split}
	&[L,\De_{\tau}+\De_{\bar{\tau}}+\De_{\la}+\De_{\bar{\la}}]|_{\ker(\De_{\tau}+\De_{\bar{\tau}}+\De_{\la}+\De_{\bar{\la}})}=0,\\
	&[\La,\De_{\tau}+\De_{\bar{\tau}}+\De_{\la}+\De_{\bar{\la}}]|_{\ker(\De_{\tau}+\De_{\bar{\tau}}+\De_{\la}+\De_{\bar{\la}})}=0.\\
	\end{split}
	\end{equation*}
\begin{lemma}\label{L3}
For any Hermitian manifold,  
	$$\ker(\la)\cap\ker(\bar{\la})\cap P^{k}\subset\ker(\De_{\tau}+\De_{\bar{\tau}}+\De_{\la}+\De_{\bar{\la}})\cap\Om^{k}$$	
	\end{lemma}
\begin{proof}
We denote by $\a$ a primitive $k$-form in $\ker(\la)\cap\ker(\bar{\la})$. By the definition of $\tau,\bar{\tau}$, we have
$$\tau\a=[\La,\la]\a=0,\ \bar{\tau}\a=[\La,\bar{\la}]\a=0.$$ 
Therefore, 
\begin{equation*}
\begin{split}
	&\tau^{\ast}\a=-\frac{i}{2}[\La,\bar{\tau}]\a=0,\ \bar{\tau}^{\ast}\a=\frac{i}{2}[\La,\tau]\a=0,\\
	&\la^{\ast}\a=-\frac{1}{3}[\La,\tau^{\ast}]\a=0,\ \bar{\la}^{\ast}\a=-\frac{1}{3}[\La,\bar{\tau}^{\ast}]\a=0.\\
	\end{split}
	\end{equation*}
	Hence, $\a\in\ker(\De_{\tau}+\De_{\bar{\tau}}+\De_{\la}+\De_{\bar{\la}})\cap\Om^{k}$.
\end{proof}
Wilison considered the following positive definite self-adjoint elliptic operator of order two:
$$\square=\De_{\pa}+\De_{\bar{\pa}}+\De_{\tau}+\De_{\bar{\tau}}+\De_{\la}+\De_{\bar{\la}}.$$
Let $\square^{p,q}$ denote the restriction to $\Om^{p,q}$.
\begin{theorem}(\cite[Theorem 3.1]{Wil20})\label{T2}
Let $(M,J,\w)$ be a compact Hermitian manifold of complex dimension $n$. For any $0\leq k\leq 2n$, there is an orthogonal direct sum decomposition
$$\ker(\square)\cap\Om^{k}=\oplus_{p+q=k}\ker({\square}^{p,q})$$
For all $0\leq p,q\leq n$, the following dualities hold:\\
(1) (Complex conjugation). We have equalities 
$$\ker({\square}^{p,q})=\overline{\ker({\square}^{q,p})}.$$
(2) (Hodge duality). The Hodge $\ast$-operator induces isomorphisms
$$\ast:\ker({\square}^{p,q})\rightarrow \ker({\square}^{n-q,n-p}).$$
(3) (Serre duality). There are isomorphisms
$$\ker({\square}^{p,q})\cong\ker({\square}^{n-p,n-q}).$$
The operator $\{L,\La,H\}$ define a finite dimensional representation of $\mathfrak{sl}(2,\C)$ on $\ker({\square})$. Moreover, for every $0\leq p,q \leq n$, the maps
$$L^{n-p-q}:\ker({\square}^{p,q})\rightarrow\ker({\square}^{n-q,n-p})$$
are isomorphisms.
\end{theorem}
The next result follows from some well established facts about $sl(2,\mathbb{C})$ representations.
\begin{corollary}(\cite[Corollary 3.2]{Wil20})\label{C2}
There is an orthogonal direct sum decomposition
$$\ker(\square^{p,q})=\bigoplus_{j\geq0}L^{j}(\ker(\square^{p-j,q-j}))_{prim}$$
where 
$$(\ker(\square^{r,s}))_{prim}:=\ker(\square^{r,s})\cap P^{r,s}.$$	
\end{corollary}
\begin{proposition}\label{P5}
Let $(M,J,\w)$ be a compact complex manifold of $dim_{\C}=n$. Then
	$$\ker(\square^{n,0})=\ker(\De_{\bar{\pa}})\cap\Om^{n,0},$$ $$\ker(\square^{0,n})=\ker(\De_{\bar{\pa}})\cap\Om^{0,n}.$$
\end{proposition}
\begin{proof}
We only proof the case of $(n,0)$-forms. We now denote by $\a^{n,0}$ a $(n,0)$-form in $\ker(\De_{\bar{\pa}})$. Therefore, $\bar{\pa}\a^{n,0}=0$. At first, we observe that $$\La\a^{n,0}=\la\a^{n,0}=\bar{\la}\a^{n,0}=0.$$
Following Lemma \ref{L3}, we have $$\Om^{n,0}=\ker(\De_{\tau}+\De_{\bar{\tau}}+\De_{\la}+\De_{\bar{\la}})\cap\Om^{n,0}.$$
It's easy to see $\pa\a^{n,0}=0$. Following the first identity on (\ref{E1}), we have
$$\pa^{\ast}\a^{n,0}=(i[\La,\bar{\pa}]-\tau^{\ast})\a^{n,0}=0.$$
Therefore, 
$$\ker(\De_{\bar{\pa}})\cap\Om^{n,0}\subset\ker(\square^{n,0})\subset\ker(\De_{\bar{\pa}})\cap\Om^{n,0},$$
that is, $\ker(\De_{\bar{\pa}})\cap\Om^{n,0}\cong\ker(\square^{n,0})$.
\end{proof}
\subsection{Morse-Novikov cohomology}
Let $M$ be a smooth manifold and $\theta$ a real valued closed one form on $M$. Define $d_{\theta}:\Om^{k}(M)\rightarrow:\Om^{k+1}(M)$ as $d_{\theta}\a=d\a+\theta\wedge\a$ for $\a\in\Om^{p}(M)$. Then we have a complex
$$\ldots\rightarrow\Om^{k-1}(M)\xrightarrow{d_{\theta}}\Om^{k}(M)\xrightarrow{d_{\theta}}\Om^{k+1}(M)\rightarrow\ldots$$
whose cohomology $H^{k}(M,\theta)=H^{k}(\Om^{\ast}(M),d_{\theta})$ is called the $k$-th Morse-Novikov cohomology group of $M$ with respect to $\theta$. $H^{k}(M,\theta)$ only depends on the de Rham cohomology class of $\theta$. This cohomology shares many properties with the ordinary de Rham cohomology (see \cite{LLMP03,Li77,OV09,Oti18}). 

We can also define an operator $d^{\ast}_{\theta}$ as the formal adjoint of $d_{\theta}$ with respect to metric $g$. Further, $\De_{\theta}=d_{\theta}d^{\ast}_{\theta}+d^{\ast}_{\theta}d_{\theta}$ is the corresponding Laplacian.  These operators are lower-order perturbations of the corresponding operators in the usual Hodge-de Rham theory. We denote by $\mathcal{H}^{k}(M,\theta)$ the space of $\De_{\theta}$-harmonic forms. The space $\mathcal{H}^{k}(M,\theta)$ is isomorphic to $H^{k}(M,\theta)$. Let $\theta^{\sharp}$ the dual vector field of $\theta$ defined by $g(\theta^{\sharp},\cdot)=\theta(\cdot)$.
\begin{lemma}(\cite[Lemma 4.4]{Chen20})\label{L1}
For any $k$-form $\a$, we have
$$\ast i_{\theta^{\sharp}}\a=(-1)^{k-1}\theta\wedge\ast\a.$$
\end{lemma}
We denote by 
$$\mathcal{S}^{k}(M,\theta)=\{\a\in\Om^{k}(M):\De_{d}\a=0, \theta\wedge\a=0, \theta\wedge\ast\a=0 \}$$
a subspace of $\mathcal{H}^{k}(M,\theta)$ on a closed manifold. We then have the following vanishing theorem.
\begin{theorem}\label{T3}
Let $M$ be a $n$-dimensional closed Riemannian manifold, $\theta$ a smooth $1$-form in $M$. If $\a$ is smooth $k$-form in $\mathcal{S}^{k}(M,\theta)$, $0\leq k\leq n$, then either $\theta=0$ or $\a=0$.
\end{theorem}
\begin{proof}
We also assume that $\a$ is a non-zero $\De_{d}$-harmonic $k$-form, i.e., $(d+d^{\ast})\a=0$. Following Lemma \ref{L1}, the equation $\theta\wedge\ast\a=0$ is equivalent to $i_{\theta^{\sharp}}\a=0$. Then we have
\begin{equation*}
0=i_{\theta^{\sharp}}(\theta\wedge\a)=|\theta|^{2}\a.\\
\end{equation*}
We denote by $Z^{c}(\a)$ the complement of the zero of $\a$.  By unique continuation of the elliptic equation $(d+d^{\ast})\a=0$, $Z^{c}(\a)$ is either empty or dense. Therefore, the vector field $\theta^{\sharp}$ is zero along $Z^{c}(\a)$. The set $Z^{c}(\a)$ is empty which is equivalent to $\a=0$ on $M$. If $Z^{c}(\a)$ is dense, then $\theta^{\sharp}=0$ almost everywhere on $M$. Since $\theta^{\sharp}$ is smooth, $\theta^{\sharp}=0$, i.e., $\theta=0$ on $M$. 
\end{proof}
\section{Harmonic symmetries on locally conformally K\"{a}hler manifold}
\subsection{Locally conformally K\"{a}hler manifold}
In this section, we first give the necessary definitions and properties of locally conformally K\"{a}hler (LCK) manifolds.
\begin{definition}
Let $(M,\w)$ be a complex Hermitian manifold of $\dim_{\C}M=n$, with
$$d\w=\theta\wedge\w,$$ 
where $\theta$ is a closed $1$-form. Then $M$ is called an LCK manifold.
\end{definition}
Therefore $\w$ is $(d-\theta)$-closed. The Morse–Novikov cohomology class $[\w]$ of $\w$ is called the Morse–Novikov class of $M$ (see \cite{LLMP03,GL84,OV09,Oti18}). This notion is similar to the notion of a K\"{a}hler class of a K\"{a}hler manifold. 
\begin{lemma}\label{L2}
	Let $(M,J,\theta)$ be a compact locally conformally K\"{a}hler manifold of $dim_{\C}=n$, $\a$ a smooth $k$-form in $\ker(\De_{\tau}+\De_{\bar{\tau}}+\De_{\la}+\De_{\bar{\la}})$. Then for any $|k-n|\geq2$, we have
	\begin{equation*}
	\theta\wedge\a=0,\ \theta\wedge\ast\a=0.
	\end{equation*}
\end{lemma}
\begin{proof}
	By the definitions of $\la$ and $\bar{\la}$, we obtain that $$\la+\bar{\la}=(\pa\w+\bar{\pa}\w)\wedge(\cdot)=\w\wedge\theta\wedge(\cdot).$$
	For simply, we let $k\leq n-2$. The case $k\geq n+2$ follows by the
	Poincare duality as the operator $\ast:\Om^{k} \rightarrow\Om^{2n-k}$ commutes with $\ker(\De_{\tau}+\De_{\bar{\tau}}+\De_{\la}+\De_{\bar{\la}})$. Since $\a\in\ker(\De_{\tau}+\De_{\bar{\tau}}+\De_{\la}+\De_{\bar{\la}})\cap\Om^{k}$, $\a$ satisfies
	$$\la\a=\bar{\la}\a=0,$$
	Therefore,  we get
	$$\w\wedge(\theta\wedge\a)=0.$$
	We then have $$L^{n-k-1}(\theta\wedge\a)=0.$$
	Since the map $L^{n-k-1}:\Om^{k+1}\rightarrow\Om^{2n-k-1}$ is bijective for $k+1\leq n-1$ (see \cite{Huy06}), we get
	$$\theta\wedge\a=0.$$
	For any $k$-form $\a$ on $M$, ($k\leq n-2$), there exists a $k$-form $\b$ such that as
	$$\ast\a=L^{n-k}\b,$$
Noting that  $(\De_{\tau}+\De_{\bar{\tau}}+\De_{\la}+\De_{\bar{\la}})(\ast\a)=0$ and $$[L,\De_{\tau}+\De_{\bar{\tau}}+\De_{\la}+\De_{\bar{\la}}]|_{\ker(\De_{\tau}+\De_{\bar{\tau}}+\De_{\la}+\De_{\bar{\la}})}=0.$$
Therefore, we have
	$$0=(\De_{\tau}+\De_{\bar{\tau}}+\De_{\la}+\De_{\bar{\la}})(\ast\a)=L^{n-k}(\De_{\tau}+\De_{\bar{\tau}}+\De_{\la}+\De_{\bar{\la}})\b$$
	Since the map $L^{n-k}:\Om^{k}\rightarrow\Om^{2n-k}$ is bijective for $k\leq n$ (see \cite{Huy06}), 
	$$(\De_{\tau}+\De_{\bar{\tau}}+\De_{\la}+\De_{\bar{\la}})\b=0.$$
By a similar way, we also have 
	$$\theta\wedge\b=0.$$
	Hence
	$$\theta\wedge\ast\a=L^{n-k}(\theta\wedge\b)=0.$$
\end{proof}	
\begin{corollary}\label{C3}	
Let $(M,J,\theta)$ be a compact locally conformally K\"{a}hler manifold of $dim_{\C}=n$. If $\theta$ is $\De_{d}$-harmonic, then for any $|k-n|\geq2$, we have
$$\ker(\De_{\tau}+\De_{\bar{\tau}}+\De_{\la}+\De_{\bar{\la}})\cap\Om^{k}=\{0\}.$$
In particular, if $(M,J,\theta)$ is a Vaisman manifold, i.e., $\na\theta=0$, then
$\ker(\De_{\tau}+\De_{\bar{\tau}}+\De_{\la}+\De_{\bar{\la}})\cap\Om^{k}=\{0\}$.
\end{corollary}
\begin{proof}
We denote by $\a$  a $k$-form in $\ker(\De_{\tau}+\De_{\bar{\tau}}+\De_{\la}+\De_{\bar{\la}})\cap\Om^{k}$. By Lemma \ref{L2}, we have $\theta\wedge\a=0$. Therefore, we have
$$i_{\theta^{\sharp}}(\theta\wedge\a)=\a|\theta|^{2}=0.$$
By unique continuation of the elliptic equation $(d+d^{\ast})\theta=0$, $Z^{c}(\theta)$ is dense or empty. Since $\theta$ is non-zero, $Z^{c}(\theta)$ is dense. Therefore, the $k$-form $\a$ is zero along $Z^{c}(\theta)$, i.e., $\a=0$ almost everywhere on $M$. Since $\a$ is smooth, $\a=0$ all over $M$. 
\end{proof}
\begin{remark}
	The Lee form $\theta$ is co-closed with respect to $g$ if only if the metric $g$ is Gauduchon, i.e., $\pa\bar{\pa}\w^{n-1}=0$ (\cite[pp. 502]{Gau84}).  A classical result of Gauduchon \cite{Gau77} states that every Hermitian metric is conformal to a Gauduchon metric, which is unique up to rescaling when $n\geq2$.
\end{remark}
The Hermitian metric on general LCK manifold could not always Gauduchon. But when we consider the space $\mathcal{H}_{\square}^{k}$, we can also prove the following vanishing theorem.
\begin{theorem}\label{T4}
	Let $(M,J,\theta)$ be a compact locally conformally K\"{a}hler manifold of $dim_{\C}=n$. Then for any $|k-n|\geq 2$, we have
	$$\mathcal{H}_{\square}^{k}=\{0\}.$$
\end{theorem}
\begin{proof}
	For any $k$-form  $\a\in\ker(\square)$ on $M$, ($|k-n|\geq2$), following Lemma \ref{L2}, we get 
	$$\theta\wedge\a=\theta\wedge\ast\a=0.$$ 
	Hence following Theorem \ref{T3}, we get either $\theta=0$ or $\a=0$. Since $M$ is non-K\"{a}hlerian, there are some points in $M$ such that $\theta\neq0$. Therefore, $\a\equiv0$, i.e., $\ker({\square})\cap\Om^{k}=\{0\}$.
\end{proof}
We now study the relationship between $\ker(\De_{\tau}+\De_{\bar{\tau}}+\De_{\la}+\De_{\bar{\la}})\cap\Om^{n-1}$ and $\ker(\De_{\tau}+\De_{\bar{\tau}}+\De_{\la}+\De_{\bar{\la}})\cap\Om^{n}$. The following result is very important in the study of Vaisman manifold.  
\begin{proposition}\label{P2}
Let $(M,J,\theta)$ be a compact locally conformally K\"{a}hler manifold of $dim_{\C}=n$. If a $(n-1)$-form $\a\in\ker(\De_{\tau}+\De_{\bar{\tau}}+\De_{\la}+\De_{\bar{\la}})$, then we have
\begin{equation*}
(\De_{\tau}+\De_{\bar{\tau}}+\De_{\la}+\De_{\bar{\la}})(\theta\wedge\a)=0,
\end{equation*}
\begin{equation*}
\theta\wedge\ast\a=0.
\end{equation*}
\end{proposition}
\begin{proof}
In order to get $(\De_{\tau}+\De_{\bar{\tau}}+\De_{\la}+\De_{\bar{\la}})(\theta\wedge\a)=0$, by Lemma \ref{L3}, we only need to prove that  $\La(\theta\wedge\a)=0$, $\la(\theta\wedge\a)=\bar{\la}(\theta\wedge\a)=0$. Since $\a\in\ker(\De_{\tau}+\De_{\bar{\tau}}+\De_{\la}+\De_{\bar{\la}})$, we have $$(\la+\bar{\la})\a=\w\wedge(\theta\wedge\a)=0,$$ 
i.e., $\La(\theta\wedge\a)=0$. We also get
\begin{equation*}
\begin{split}
&\la(\theta\wedge\a)=\theta^{1,0}\wedge\w\wedge(\theta\wedge\a)=0,\\
&\bar{\la}(\theta\wedge\a)=\theta^{0,1}\wedge\w\wedge(\theta\wedge\a)=0.
\end{split}
\end{equation*}
For any $(n-1)$-form $\a$ on $M$, there exists a $(n-1)$-form $\b$ such that
$$\ast\a=L\b.$$ 
By a similar way in Lemma \ref{L2}, we get
$$(\De_{\tau}+\De_{\bar{\tau}}+\De_{\la}+\De_{\bar{\la}})\b=0.$$
We then have 
$$\la(\b)=\w\wedge\theta^{1,0}\wedge\b=0,\  \bar{\la}(\b)=\w\wedge\theta^{0,1}\wedge\b=0.$$
Hence
$$0=(\la+\bar{\la})\b=\w\wedge(\theta\wedge\b)=\theta\wedge\ast\a.$$
\end{proof}
The next result is in regards to the $k$-forms $(k=n,n-1)$ in $ker(\square)$.
\begin{theorem}\label{P4}
Let $(M,J,\theta)$ be a compact locally conformally K\"{a}hler manifold of $dim_{\C}=n$. We then have
\begin{equation*}
\ker\De_{\bar{\pa}}\cap P^{k,n-1-k}\cap\ker(i_{\theta^{\sharp}})\cong\ker(\square^{k,n-1-k}).
\end{equation*}
$$\ker\De_{\bar{\pa}}\cap P^{k,n-k}\cong \ker(\square^{k,n-k}).$$
\end{theorem}
\begin{proof} 
First, we prove that 
$$P^{k,n-1-k}\cap\ker(i_{\theta^{\sharp}})\subset\ker(\De_{\tau}+\De_{\bar{\tau}}+\De_{\la}+\De_{\bar{\la}})\cap\Om^{k,n-1-k}.$$
We denote by $\a$ a  primitive $(k,n-1-k)$-form. Hence we have (\cite[Proposition 1.2.31]{Huy06})
$$\ast\a=(-1)^{\frac{n(n-1)}{2}}(\sqrt{-1})^{2k+1-n}(L\a).$$
Since $\a\in\ker(i_{\theta^{\sharp}})$, we get
$$0=i_{\theta^{\sharp}}\a=\ast(\theta\wedge\ast\a)=\ast(\theta\wedge L\a),$$
i.e.,
$$0=\theta^{1,0}\wedge L\a=\la(\a),\ 0=\theta^{0,1}\wedge L\a=\bar{\la}(\a).$$
Following Lemma \ref{L3}, we obtain that 
$$\a\in\ker(\De_{\tau}+\De_{\bar{\tau}}+\De_{\la}+\De_{\bar{\la}})\cap\Om^{k,n-1-k}.$$
Noting that (see \cite[Chapter VI, Corollary 6.15]{Dem97}, \cite[Equation (9)]{Wil} )
$$\De_{\pa}+[\pa,\tau^{\ast}]=\De_{\bar{\pa}}+[\bar{\pa},\bar{\tau}^{\ast}] .$$
We also assume that $\a\in\ker\De_{\bar{\pa}}$. Then we observe that
\begin{equation*}
\begin{split}
&([\pa,\tau^{\ast}]\a,\a)=(\tau^{\ast}\pa\a,\a)=(\pa\a,\tau\a)=0,\\
&([\bar{\pa},\bar{\tau}^{\ast}]\a,\a)=(\bar{\tau}^{\ast}\bar{\pa}\a,\a)=(\bar{\pa}\a,\bar{\tau}\a)=0.\\
\end{split}
\end{equation*}
Combing above identities, we get
$$(\De_{\pa}\a,\a)=0,$$
i.e., $\De_{\pa}\a=0$. Therefore,
$$\ker\De_{\bar{\pa}}\cap P^{k,n-1-k}\cap\ker(i_{\theta^{\sharp}})\subset\ker(\square^{k,n-1-k}).$$
On the other hand, we let $\a\in\ker(\square^{k,n-1-k})$. Following Corollary \ref{C2}, we get
$$\a=\a_{0}+\sum_{j\geq1}L^{j}\a_{j},$$
where $\a_{i}$ is a primitive $(k-i,n-1-k+i)$-form in $\ker(\square)$ for all $i\geq0$. Following vanishing theorem \ref{T4}, we get $\a_{j}=0$ for all $j\geq1$. Therefore, $\a$ is primitive. Following Proposition \ref{P2}, we also have $\theta\wedge\ast\a=0$. Therefore,
$$\ker(\square^{k,n-1-k})\subset\ker\De_{\bar{\pa}}\cap P^{k,n-1-k}\cap\ker(i_{\theta^{\sharp}}).$$
Hence, we get $\ker\De_{\bar{\pa}}\cap P^{k,n-1-k}\cap\ker(i_{\theta^{\sharp}})\cong \ker(\square^{k,n-1-k})$.

Next, we consider the case of $(k,n-k)$-forms. We denote by $\a$ a $(k,n-k)$-form in $\ker\De_{\bar{\pa}}\cap P^{k,n-k}$. Noting that $\La\a=0$, i.e., $L\a=0$. Therefore,
\begin{equation*}
\la\a=\theta^{1,0}\wedge L\a=0,\ \bar{\la}\a=\theta^{0,1}\wedge L\a=0.
\end{equation*}
Following Lemma \ref{L3}, we also have$(\De_{\tau}+\De_{\bar{\tau}}+\De_{\la}+\De_{\bar{\la}})\a=0$. Using the identities in (\ref{E1}), we get 
\begin{equation*}
\begin{split}
&\pa^{\ast}\a=(i[\La,\bar{\pa}]-\tau^{\ast})\a=0,\\ 
& \pa\a=(i[L,\bar{\pa}^{\ast}]-\tau)\a=0.
\end{split}
\end{equation*}
Therefore, $$\ker\De_{\bar{\pa}}\cap P^{k,n-k}\subset \ker(\square^{k,n-k}).$$
On the other hand, we let $\a\in\ker(\square^{k,n-k})$. By a similar way, we also obtain that $\a$ is primitive. Therefore,
$$\ker(\square^{k,n-k})\subset\ker\De_{\bar{\pa}}\cap P^{k,n-k}.$$
Hence, we get $\ker\De_{\bar{\pa}}\cap P^{k,n-k}\cong \ker(\square^{k,n-k})$.
\end{proof}
\subsection{Vaisman manifold}
Among the LCK manifolds, a distinguished class is the following. 
\begin{definition}(\cite{Vai82} and \cite[Definition 3.7]{Ver04})
An LCK manifold $(M,J,\theta)$ is called Vaisman if $\na\theta=0$, where $\na$ is the Levi–Civita connection of the metric $g(\cdot,\cdot)=\w(J\cdot,\cdot)$. If $\theta\neq0$, then after rescaling, we may always assume that $|\theta|=1$.  
\end{definition}
Before proof our results, we recall the decomposition of harmonic forms on a compact Vaisman manifold \cite{Tsu94,Vai82}. We denote by $\theta^{\sharp}$ (resp. $(J\theta)^{\sharp}$ the dual vector field of $\theta$ (resp. $J\theta$) with respect to metric $g$. Let $\mathcal{D}^{1}$ (resp. $\mathcal{D}^{2}$) be the $1$-dimensional distribution spanned by the Lee field $\theta^{\sharp}$  (resp. by the anti-Lee field $(J\theta)^{\sharp}$). We set $\mathcal{V}=\mathcal{D}^{1}\oplus\mathcal{D}^{2}$ and call it the vertical foliation. In any foliated chart the metric of $M$ can be expressed as
\begin{equation}\label{E4}
ds^{2}=g_{a\bar{b}}dz^{a}\otimes d\bar{z}^{b}+(\theta-iJ\theta)\otimes(\theta+iJ\theta).
\end{equation}
The direct sum decomposition $TM=\mathcal{V}\oplus\mathcal{V}^{\perp}$ produces a corresponding decomposition of the differential forms on $M$ into sums of bihomogeneous forms of type $(p,q)$, where $p$ is the transversal degree and $q$ the leaf degree. This, moreover, decomposes the exterior differentiation operator as
\begin{equation}\label{E5}
d=d'+d''+\pa
\end{equation}
where $d'$ has type $(1,0)$, $d''$ has type $(0,1)$ and $\pa$ has type $(2,-1)$. The Hodge $\ast$-operator of $(M,g)$ acts homogeneously and (\ref{E5}) implies a decomposition of the corresponding adjoint operators
\begin{equation}\label{E6}
\de=\de'+\de''+\tilde{\pa},
\end{equation}
where $\de,\de',\de'',\tilde{\pa}$ are the adjoint operator of $d,d',d'',\pa$ respectively. If $\ast'$ denotes the Hodge $\ast$ of the transversal part of the metric $g$ of $M$ given by (\ref{E4}), we have
$$\ast\a=-J\theta\wedge \theta\wedge\ast'\a,$$
where $\a$ is a $\mathcal{V}$-foliation form. 

We also denote by $\w':=-dJ\theta=2i\pa\theta^{0,1}$ form of transverse (K\"{a}hlerian) part of the metric. For simple set $L'=\w'\wedge\cdot$ and $\La'=i_{\w'}$. Let us define
$$S^{k}(\mathcal{V})=\{\a\in\Om^{k}(\mathcal{V}):\De'\a=0,\ \La'\a=0 \},\ k\leq n-1$$
$$S^{p,q}(\mathcal{V})=\{\a\in\Om^{p,q}(\mathcal{V}):\De'\a=0,\ \La'\a=0 \},\ p+q\leq n-1,$$
where $\De'=d'\de'+\de'd'$, $\Om^{k}(\mathcal{V})$ (resp. $\Om^{p,q}(\mathcal{V})$) is the set of $\mathcal{V}$-foliate $k$-forms (resp. $(p,q)$-forms) and denote $s_{k}=\dim\mathcal{S}^{k}(\mathcal{V})$ (resp.  $s_{p,q}=\dim\mathcal{S}^{p,q}(\mathcal{V})$). Clearly $$S^{k}(\mathcal{V})=\bigoplus_{p+q=k}S^{p,q}(\mathcal{V}).$$ 
We denote by $H^{p,q}_{\bar{\pa}}(M)$ the Dolbeault cohomology group of type $(p,q)$ and put $h^{p,q}(M)=\dim H^{p,q}_{\bar{\pa}}(M)$.
\begin{theorem}(\cite[Theorem 4.1]{Vai82})\label{T6}
Let $M$ be a compact Vaisman manifold of $dim_{\C}=n$. Then, an $k$-form $\a$ of $M$ with
$0\leq k\leq n-1$ is $\De_{d}$-harmonic iff 
$$\a=\b+\theta\wedge\gamma,$$
where $\b$, $\gamma$ are transversally harmonic and transversally effective foliate forms. In particular,
$$H^{k}(M)\cong S^{k}(\mathcal{V})\oplus S^{k-1}(\mathcal{V}),\ k\leq n-1.$$
\end{theorem}
It is known that on any Vaisman manifold, the following formula holds, \cite{DO98,Vai82} 
\begin{equation}
\w=\theta\wedge J\theta-dJ\theta.
\end{equation} 
We then have
\begin{lemma}\label{L4}
	Let $(M,J,\theta)$ be a compact Vaisman manifold of $dim_{\C}=n$. If a $k$-form $\a$ with $0\leq k\leq n-1$ is $\De_{d}$-harmonic, then 
	\begin{equation*}
	\La(\theta\wedge\a)=0.
	\end{equation*}	
\end{lemma}
\begin{proof}
Following Theorem \ref{T6}, there exist two forms $\b$, $\gamma$ such that
$$\a=\b+\theta\wedge\gamma,$$ 
where $\b,\gamma$ are transversally harmonic and transversally effective foliate forms. Therefore, 
	\begin{equation*}
	\begin{split}
	\La(\theta\wedge\a)&=\La(\theta\wedge\b)\\
	&=\ast(\w\wedge\ast(\theta\wedge\b)\\
	&=(\pm)\ast(\w\wedge J\theta\wedge\ast'\b)\\
	&=(\pm)\ast(dJ\theta\wedge J\theta\wedge\ast'\b).
	\end{split}
	\end{equation*}
	Noticing that $\b$ satisfies $\La'\b=0$, i.e., $dJ\theta\wedge\ast'\b=0$. We then have $\La(\theta\wedge\a)=0$.
\end{proof}
We now want to obtain a complex version of Theorem \ref{T6}. As usual we let $\De_{\bar{\pa}}=\bar{\pa}^{\ast}\bar{\pa}+\bar{\pa}\bar{\pa}^{\ast}$ be the complex Laplacian. 
\begin{theorem} \label{T7}(\cite[Theorem 3.2]{Tsu94})
	Let $(M,J,\theta)$ be a compact Vaisman manifold of $dim_{\C}=n$. Then any $(p, q)$-form $\a$ on $M$, $0\leq p+q \leq n-1$ satisfies $\De_{\bar{\pa}}\a=0$ iff 
	$$\a=\b+\theta^{0,1}\wedge\gamma,$$ 
	where $\b$ and $\gamma$ are transversally harmonic and transversally effective $\mathcal{V}$-foliate forms. In particular,
	$$H^{p,q}_{\bar{\pa}}(M)\cong S^{p,q}(\mathcal{V})\oplus S^{p,q-1}(\mathcal{V}),\ k\leq n-1.$$
\end{theorem}
\begin{theorem}(\cite[Theorem 3.5]{Tsu94})\label{T10}
	On a compact Vaisman manifold $M$ of $dim_{\C}=n$,  for any $0\leq k\leq 2n$, we have 
	$$b^{k}(M)=\sum_{p+q=k}h^{p,q}(M).$$
Furthermore,
$$b^{n}(M)=2s_{n-1}.$$
\end{theorem}
\begin{proof}
Following Theorem \ref{T6} and \ref{T7}, we get
$$H^{k}(M)\cong\bigoplus_{p+q=k}S^{p,q}(\mathcal{V})\oplus\bigoplus_{p+q=k}S^{p,q-1}(\mathcal{V})\cong \bigoplus_{p+q=k}H^{p,q}_{\bar{\pa}}(M).$$
	Therefore, $b^{k}=\sum_{p+q=k}h^{p,q}$ for $k\leq n-1$. By Poincar\'{e} and Serre duality it is also true for $k\geq n+1$. As for $k=n$, one uses the following formula for the Euler characteristic:
	$$\chi(M)=\sum_{k=0}^{2n}(-1)^{k}b^{k}(M)=\sum_{p+q=0}^{2n}(-1)^{p+q}h^{p,q}(M).$$
Noting that the Euler number $\chi(M)$ of compact Vaisman manifold is zero. Since $b^{k}=s_{k}+s_{k-1}$ for $k\leq n-1$ and $b_{k}=b_{2n-k}$, we then have
\begin{equation*}
\begin{split}
0&=\sum_{k=0}^{2n}(-1)^{k}b^{k}\\
&=2\sum_{k\leq n-1}(-1)^{k}b^{k}+(-1)^{n}b^{n}\\
&=2\sum_{k\leq n-1}(-1)^{k}s_{k}-2\sum_{k\leq n-2}(-1)^{k}s_{k}+(-1)^{n}b^{n}\\
&=(-1)^{n}(b^{n}-2s_{n-1}).
\end{split}
\end{equation*}
Hence we get $b^{n}=2s_{n-1}$.
\end{proof}
\begin{proposition}\label{P3}
	Let $(M,J,\theta)$ be a compact Vaisman manifold of $dim_{\C}=n$. If a $(k,n-k-1)$-form $\a\in\ker(\square^{k,n-1-k})$, then 
	\begin{equation*}
	\square(\theta\wedge\a)=0.
	\end{equation*}
	In particular,
	$$\square(\theta^{1,0}\wedge\a)=0,\ \square(\theta^{0,1}\wedge\a)=0.$$
\end{proposition}
\begin{proof}
	Noting that $(\De_{\tau}+\De_{\bar{\tau}}+\De_{\la}+\De_{\bar{\la}})\a=0$.
	Following Proposition \ref{P2}, we get
$$(\De_{\tau}+\De_{\bar{\tau}}+\De_{\la}+\De_{\bar{\la}})(\theta\wedge\a)=0.$$
Hence we have
	$$\tau^{\ast}(\theta\wedge\a)=\bar{\tau}^{\ast}(\theta\wedge\a)=0,$$
	$$\tau(\theta\wedge\a)=\bar{\tau}(\theta\wedge\a)=0.$$
Now we begin to prove that
	$$(\De_{\pa}+\De_{\bar{\pa}})(\theta\wedge\a)=0.$$ 
	Noting that $\na\theta=0$. Following \cite[Corollary 3.6]{HT20} and \cite[Proposition 2.5 and Corollary 2.9]{Ver11}, for any $\a\in\ker(\square^{k,n-1-k})\subset\ker(\De_{d})$, we have	
	$$\De_{d}(\theta\wedge\a)=0,$$
	i.e., $d(\theta\wedge\a)=d^{\ast}(\theta\wedge\a)=0$. Following the identities in (\ref{E2}), we have
	\begin{equation*}
	\begin{split}
	([\La,\bar{\pa}]-[\La,\pa])(\theta\wedge\a)&=-i(\pa^{\ast}+\tau^{\ast})(\theta\wedge\a)-i(\bar{\pa}^{\ast}+\bar{\tau}^{\ast})(\theta\wedge\a)\\
	&=-id^{\ast}(\theta\wedge\a)=0.\\
	\end{split}
	\end{equation*}	
	Following $\La(\theta\wedge\a)=0$ (see Lemma \ref{L4}) and  the fact $0=d\theta=\pa\theta^{0,1}+\bar{\pa}\theta^{1,0}$, $\pa\theta^{1,0}=\bar{\pa}\theta^{0,1}=0$, we then have
	\begin{equation*}
	\begin{split}
	([\La,\bar{\pa}]-[\La,\pa])(\theta\wedge\a)&=\La((\bar{\pa}-\pa)\theta\wedge\a)\\
	&=\La((\bar{\pa}\theta^{1,0}-\pa\theta^{0,1})\wedge\a)\\
	&=2\La(\bar{\pa}\theta\wedge\a)\\
	&=2[\La,\bar{\pa}](\theta\wedge\a).
	\end{split}
	\end{equation*}
	Therefore, we have
	$$[\La,\bar{\pa}](\theta\wedge\a)=[\La,\pa](\theta\wedge\a)=0.$$
	Following the identities in (\ref{E1}) and Proposition \ref{P2}, we have
	\begin{equation*}
	\begin{split}
	&\pa^{\ast}(\theta\wedge\a)=i([\La,\bar{\pa}]-\tau^{\ast})(\theta\wedge\a)=0 ,\\
	&\bar{\pa}^{\ast}(\theta\wedge\a)=-i([\La,\pa]-\bar{\tau}^{\ast})(\theta\wedge\a)=0.\\
	\end{split}
	\end{equation*}
	Next, using the identities in (\ref{E1}) again,  we have
	\begin{equation*}
	\begin{split}
	&\pa(\theta\wedge\a)=i([L,\bar{\pa}^{\ast}]-\tau)(\theta\wedge\a)=0 ,\\
	&\bar{\pa}(\theta\wedge\a)=-i([L,\pa^{\ast}]-\bar{\tau})(\theta\wedge\a)=0.\\
	\end{split}
	\end{equation*}
	Here we use the fact $\la(\a)=\w\wedge\theta^{1,0}\wedge\a=0$, $\bar{\la}(\a)=\w\wedge\theta^{0,1}\wedge\a=0$. Hence, we have
	$$\square(\theta\wedge\a)=0.$$ 
Therefore, $\square(\theta^{1,0}\wedge\a)=0,\ \square(\theta^{0,1}\wedge\a)=0$.
\end{proof}
We now study the spaces $\ker(\square^{p,n-p-1})$ $(0\leq p\leq n-1)$. We prove that $\ker(\square^{p,n-p-1})$ is actually $\mathcal{S}^{p,n-1-p}(\mathcal{V})$. 
\begin{theorem}\label{T8}
	Let $(M,J,\theta)$ be a compact Vaisman manifold of $dim_{\C}=n$. Then 
	$$\ker(\square^{p,n-1-p})\cong\mathcal{S}^{p,n-1-p}(\mathcal{V}),$$
	where $\mathcal{S}^{p,q}$ denote the vector space of transversally harmonic and transversally effective $\mathcal{V}$-foliate $(p,q)$-forms. 
\end{theorem}
\begin{proof}
Following Theorem \ref{P4}, we only need prove that 
$$\ker\De_{\bar{\pa}}\cap P^{p,n-1-p}\cap\ker(i_{\theta^{\sharp}})\cong\mathcal{S}^{p,n-1-p}(\mathcal{V}).$$	
We denote by $\a$ an $\mathcal{V}$-foliate transversally $(p,n-1-p)$-form in $\mathcal{S}^{p,n-p-1}$. Noting that $\La'\a=0$, i.e.,
$$0=dJ\theta\wedge\ast'\a,$$
and $i_{\theta^{\sharp}}\a=0$.  Since $\w=\theta\wedge J\theta-dJ\theta$, we get
\begin{equation*}
\begin{split}
\La\a&=-\ast^{-1}L\ast\a\\
&=-\ast^{-1}(\w\wedge\ast'\a\wedge\theta\wedge J\theta)\\
&=-\ast^{-1}(dJ\theta\wedge\ast'\a\wedge\theta\wedge J\theta)=0.\\
\end{split}
\end{equation*}
Therefore, $\mathcal{S}^{p,n-1-p}(\mathcal{V})\subset\ker\De_{\bar{\pa}}\cap P^{p,n-1-p}\cap\ker(i_{\theta^{\sharp}})$.

On the other hand,	we denote by $\a$ a $(p,n-1-p)$-form in $\ker(\square^{p,n-1-p})$. Following Theorem \ref{T7}, we obtain that 
	$$\a=\b+\theta^{0,1}\wedge\gamma,$$
	where $\b,\gamma$ are transversally effective $\mathcal{V}$-foliation forms. Following Proposition \ref{P2}, we have 
	$$0=i_{\theta^{\sharp}}\a=|\theta^{0,1}|^{2}\gamma=\frac{1}{2}\gamma,$$
	i.e., $\gamma=0$. Therefore, $\a\in\mathcal{S}^{p,n-p-1}$.
\end{proof}
\begin{corollary}(\cite[Corollay 3.4]{Tsu94})\label{C1}
On a compact Vaisman manifold $(M,J,\theta)$, there exists the isomorphism
$$H^{n,0}_{\bar{\pa}}(M)\cong H^{n-1,0}_{\bar{\pa}}(M)$$
where$H^{p,q}_{\bar{\pa}}(M)$ is the Dolbeault cohomology group of type $(p, q)$. Furthermore,
$$H^{n,0}_{\bar{\pa}}(M)\cong S^{n,0}(\mathcal{V})\cong H^{n-1,0}_{\bar{\pa}}(M)\cong S^{n-1,0}(\mathcal{V}).$$
\end{corollary}
We begin to study the spaces $\ker(\square^{p,n-p})$ $(0\leq p\leq n)$.
\begin{theorem}\label{T9}
Let $(M,J,\theta)$ be a compact Vaisman manifold of $dim_{\C}=n$. Then we have
$$h_{\square}^{p,n-p}=s_{p,n-p-1}+s_{p-1,n-p}$$
where $s_{k}=\dim\mathcal{S}^{k}$ (resp.  $s_{p,q}=\dim\mathcal{S}^{p,q}$). Furthermore,
\begin{equation*}
\ker\De_{d}\cap\Om^{n}\cong\ker(\square)\cap\Om^{n},
\end{equation*}
\begin{equation*}
\ker(\square^{k,n-k})\cong\ker\De_{\bar{\pa}}\cap\Om^{k,n-k}.
\end{equation*}
\end{theorem}
\begin{proof}
Let $\a_{1}$ be a non-zero $(p,n-p-1)$-form in $\mathcal{S}^{p,n-p-1}$ and $\a_{2}$ a non-zero $(p-1,n-p)$-form in $\mathcal{S}^{p-1,n-p}$ . Following Proposition \ref{P3},  $\theta^{1,0}\wedge\a_{1}$ and $\theta^{0,1}\wedge\a_{2}$ are two non-zero differential forms  in $\ker(\square^{p,q})$. Therefore,
$$h_{\square}^{p,n-p}\geq s_{p,n-p-1}+s_{p-1,n-p}.$$
It's easy to see $\ker(\square)\cap\Om^{n}\subset\ker\De_{d}\cap\Om^{n}$. Following Theorem \ref{T10}, we then have
$$b^{n}=2s_{n-1}\geq\sum_{p=0}^{n}h_{\square}^{p,n-p}.$$ 
Therefore,
$$2s_{n-1}\geq\sum_{p=0}^{n}s_{p,n-p-1}+\sum_{p=0}^{n}s_{p-1,n-p}=2s_{n-1}.$$
Hence, all inequalities must be equalities. We obtain that
$$h_{\square}^{p,n-p}=s_{p,n-p-1}+s_{p-1,n-p},$$
i.e., 
if $\a\in\ker(\square^{p,n-p})$, then there are $\b_{1}\in\mathcal{S}^{p-1,n-p}$ and $\b_{2}\in\mathcal{S}^{p,n-p-1}$ such that
$$\a=\theta^{1,0}\wedge\b_{1}+\theta^{0,1}\wedge\b_{2},$$
Following $\ker(\square)\cap\Om^{n}=\oplus_{0\leq k\leq n}\ker({\square}^{k,n-k})$ and $b^{n}=2s_{n-1}=\sum_{k=0}^{n}h_{\square}^{k,n-k}$, we then have
\begin{equation*}
\ker\De_{d}\cap\Om^{n}\cong\ker(\square)\cap\Om^{n},
\end{equation*}
Following Theorem \ref{T10}, we get
$$b^{n}=\sum_{k=0}^{n}h^{k,n-k}=\sum_{k=0}^{n}h_{\square}^{k,n-k}=h_{\square}^{n}.$$
We then have
$$h^{k,n-k}=h_{\square}^{k,n-k},$$
i.e., 
$$\ker(\square^{k,n-k})\cong\ker\De_{\bar{\pa}}\cap\Om^{k,n-k}.$$
\end{proof}
\begin{corollary}\label{C4}
Let $(M,J,\theta)$ be a compact Vaisman manifold of $dim_{\C}=n\geq2$.\\
	(1) If the Betti number $b^{n}=0$, then for any $0\leq k\leq 2n$, we have
	$$\mathcal{H}_{\square}^{k}=\{0\}.$$
	(2) If the Betti number $b^{n}=2$, then 
	$$\mathcal{H}^{n,0}_{\bar{\pa}}(M)=\mathcal{H}^{n-1,0}_{\bar{\pa}}(M)=\{0\}.$$
\end{corollary}
\begin{proof}
If $b^{n}=0$, then
$$\mathcal{H}_{\square}^{n}=\mathcal{H}_{\square}^{n-1}=0.$$
If $b^{n}=2$, then following Corollary \ref{C1} and Theorem \ref{T10}, we have $$2h^{n-1,0}\leq 2s_{n-1}=b^{n}.$$
Therefore,
$h^{n-1,0}(M)\leq1$. When $h^{n-1,0}(M)=1$, we denote by $\a$ a non-zero $(n-1,0)$-form in $\ker(\square^{n-1,0})$. Following Theorem \ref{T8}, $\a\in\mathcal{S}^{n-1,0}(\mathcal{V})$. Therefore, $\theta^{1,0}\wedge\a$ and  $\theta^{0,1}\wedge\a$ are not zero. Then by Proposition \ref{P3}, we get  $\theta^{1,0}\wedge\a\in\ker(\square^{n,0})$,  $\theta^{0,1}\wedge\a\in\ker(\square^{n-1,1})$. Therefore, we get
$$h_{\square}^{n,0}=h_{\square}^{n-1,0}=1,\  h_{\square}^{n-1,1}\geq h_{\square}^{n-1,0}=1.$$ 
Therefore, 
$$b^{n}=\sum_{p+q=n}h_{\square}^{p,q}\geq 2h_{\square}^{n,0}+2h_{\square}^{n-1,1}\geq4.$$
This contradicts the fact that $b^{n}=2$. Hence $h^{n-1,0}(M)=0$.
\end{proof}
\begin{remark}
For larger number $b^{n}$, we cannot prove that $s_{n-1,0}$ must be zero. For example, when $n=3$, $b^{3}=4$, i.e., $s_{2}=h^{1,1}_{\square}+2h^{2,0}_{\square}=2$, we get
$$b^{3}=2h^{3,0}_{\square}+2h^{2,1}_{\square}\geq 4h^{2,0}_{\square}.$$
Therefore, $h^{2,0}_{\square}=0$ or $1$. There are two cases as follows:\\
(i) $h_{\square}^{1,1}=2$, $h_{\square}^{0,2}=h_{\square}^{2,0}=0$, $h_{\square}^{2,1}=h_{\square}^{1,2}=2$ and $h_{\square}^{0,3}=h_{\square}^{3,0}=0$;\\
(ii)$h_{\square}^{1,1}=0$, $h_{\square}^{0,2}=h_{\square}^{2,0}=1$, $h_{\square}^{2,1}=h_{\square}^{1,2}=1$ and $h_{\square}^{0,3}=h_{\square}^{3,0}=1$.\\
\end{remark}
\section*{Acknowledgements}
The author thanks S.O.Wilson for his useful comments and suggestions, which enhance the quality of this article. I would also like to thank the anonymous referee for careful reading of my manuscript and helpful comments. This work was supported in part by NSF of China (11801539) and the Fundamental Research Funds of the Central Universities (WK3470000019), the USTC Research Funds of the Double First-Class Initiative (YD3470002002).

\end{document}